\documentclass[a4paper,12pt]{article}     
\usepackage{amsmath,amscd,amssymb}        
\usepackage{latexsym}                     
\usepackage[english, german]{babel}                

\usepackage{theorem}                 

\newtheorem{theorem}{Theorem}
\newtheorem{corollary}{Corollary}
\newtheorem{proposition}{Proposition}
\newtheorem{lemma}{Lemma}

\newenvironment{remark}
{\smallskip\noindent{\bf Remark\/}.}{\smallskip\par}

\newenvironment{proof}{\begin{ProofwCaption}{Proof}}{\end{ProofwCaption}}
\newenvironment{proof*}[1]{\begin{ProofwCaption}{{#1}}}{\end{ProofwCaption}}
\newenvironment{ProofwCaption}[1]%
  {\addvspace\theorempreskipamount \noindent{\it #1.}\rm}%
  {\qed \par \addvspace\theorempostskipamount}
\newcommand{\qedsymbol}{\mbox{$\Box$}}
\newcommand{\qed}{\hfill\qedsymbol}

\newcommand{\CC}{{\mathbb C}}

\newcommand{\PP}{{\mathbb P}}

\newcommand{\RR}{{\mathbb R}}
\newcommand{\ZZ}{{\mathbb Z}}

\newcommand{\calO}{{\cal O}}

\newcommand{\calE}{{\cal E}}

\newcommand{\calD}{{\cal D}}

\newcommand{\uu}{\underline{u}}
\newcommand{\vv}{\underline{v}}
\newcommand{\1}{\underline{1}}
\newcommand{\uupsilon}{{\underline{\upsilon}}}

\newcommand{\mm}{\underline{m}}
\newcommand{\sss}{\underline{s}}
\newcommand{\ttt}{\underline{t}}
\newcommand{\gggg}{\underline{g}}
\newcommand{\MM}{{\underline M}}
\newcommand{\NNN}{{\underline N}}

\title{On a Newton filtration for functions on a curve singularity}
\author{W.~Ebeling and S.~M.~Gusein-Zade
\thanks{Partially supported by the DFG (Eb 102/7--1), the Russian government grant 11.G34.31.0005, RFBR--10-01-00678,
NSh--4850.2012.1 and Simons--IUM fellowship.
AMS 2010 Math. Subject Classification: 32S05, 14M25, 16W70.
}
}
\date{}

\begin{document}
\selectlanguage{english}

\maketitle

\begin{abstract}
In a previous paper, there was defined a multi-index filtration on the ring of functions on a hypersurface singularity corresponding to its Newton diagram generalizing (for a curve singularity) the divisorial one. Its Poincar\'e series was computed for plane curve singularities non-degenerate with respect to their Newton diagrams. Here we use another technique to compute the Poincar\'e series for plane curve singularities without the assumption that they are non-degenerate with respect to their Newton diagrams. We show that the Poincar\'e series only depends on the Newton diagram and not on the defining equation. 
\end{abstract}

\section*{Introduction}
In \cite{EG1, EG2}, there were defined two multi-index filtrations on the ring $\calO_{\CC^n,0}$
of germs of holomorphic functions in $n$ variables associated to a Newton diagram $\Gamma$ in $\RR^n$
and to a germ of an analytic function $f : (\CC^n,0) \to (\CC,0)$ with this Newton diagram. We assumed that the function $f$
was non-degenerate with respect to its Newton diagram $\Gamma$. These filtrations are essentially filtrations
on the ring $\calO_{V,0}=\calO_{\CC^n,0}/(f)$ of germs of functions on the hypersurface singularity $V=\{f=0\}$.
They correspond to the quasihomogeneous valuations on the ring $\calO_{\CC^n,0}$ defined by the facets of 
the diagram $\Gamma$. These facets correspond to some components of the exceptional divisor of a toric resolution
of the germ $f$ constructed from the diagram $\Gamma$. Such a component defines the corresponding divisorial
valuation on the ring $\calO_{\CC^n,0}$. For $n\ge 3$ (and for a $\Gamma$-non-degenerate $f$) these valuations induce
divisorial valuations on the ring $\calO_{V,0}$ and define the corresponding multi-index filtration on it.
The filtration defined in \cite{EG1} was regarded as a certain ``simplification'' of the divisorial one.
This appeared not to be the case. For example, a general formula for the Poincar\'e series of this filtration
is not known even for the number of variables $n=2$. For Newton diagrams of special type, A.~Lemahieu
identified this filtration with a so called embedded filtration on $\calO_{V,0}$ \cite{L}.
In \cite{L}, a formula for the Poincar\'{e} series of the embedded filtration for a hypersurface singularity was given.
H.~Hamm studied the embedded filtration and the corresponding Poincar\'{e} series for complete intersection singularities \cite{H}.

In \cite{EG2}, there was given an ``algebraic'' definition of the divisorial valuation corresponding to
a Newton diagram (for $n\ge 3$) somewhat similar to the definition in \cite{EG1}. Roughly speaking,
the difference consists in using the ring $\calO_{\CC^n,0}[x_1^{-1}, \ldots, x_n^{-1}]$ instead of $\calO_{\CC^n,0}$.
For $n=2$, this definition does not give, in general, a valuation, but an order function (see the definition below).
For a $\Gamma$-non-degenerate $f\in \calO_{\CC^2,0}$, this order function was described as a
``generalized divisorial valuation'' defined by the divisorial valuations corresponding to all the points
of intersection of the resolution (normalization) $\widetilde{V}$ of the curve $V$ with the corresponding component
of the exceptional divisor. This permitted to apply the technique elaborated in \cite{IJM} and to compute the
corresponding Poincar\'e series. (This technique has no analogue which could
be applied to degenerate $f$, or to the case $n>2$, or to the filtration defined in \cite{EG1}.)
In this case the Poincar\'e series depends only on the Newton diagram $\Gamma$
and does not depend on the function $f$ with $\Gamma_f=\Gamma$.

The definitions in \cite{EG1} and \cite{EG2} make also sense for functions $f$ degenerate with respect
to their Newton diagrams. Here we compute the Poincar\'e series of the filtration introduced in \cite{EG2}
for $n=2$ directly from the definition without the assumption that
$f$ is non-degenerate with respect to the Newton diagram.
We show that the answer is the same as in \cite[Corollary 1]{EG2} for non-degenerate $f$.
Thus, for $n=2$, the Poincar\'e series of this filtration depends only on the Newton diagram $\Gamma$.
One can speculate that the same holds for $n\ge 3$ and for the Poincar\'e series of the filtration defined in \cite{EG1}.

We hope that some elements of the technique used here can be applied to the case $n\ge 3$ and/or to
the filtration defined in \cite{EG1} as well.

One motivation to study (multi-variable) Poincar\'e series of filtrations comes from the fact that they are sometimes related or even coincide with appropriate monodromy zeta functions or with Alexander polynomials (see e.g.\ \cite{IJM}).
We show that the obtained formula for the Poincar\'e series has a relation to the (multi-variable) Alexander polynomial of a collection of functions.

\section{Filtrations associated to Newton diagrams}
Let $(V,0)$ be a germ of a complex analytic variety and let $\calO_{V,0}$ be the ring of
germs of holomorphic functions on $(V,0)$. A map $v: \calO_{V,0}\to \ZZ_{\ge 0}\cup \{+\infty\}$
is an {\em order function} on $\calO_{V,0}$ if $v(\lambda g)=v(g)$ for a non-zero $\lambda \in \CC$ and
$v(g_1+g_2)\ge \min\{v(g_1), v(g_2)\}$. (If, moreover, $v(g_1g_2)=v(g_1)+v(g_2)$,
the map $v$ is a {\em valuation} on $\calO_{V,0}$.) A collection $\{v_1, v_2, \ldots, v_r\}$
of order functions on $\calO_{V,0}$ defines a multi-index filtration on $\calO_{V,0}$:
\begin{equation}\label{filtration}
J(\uupsilon):=\{g\in \calO_{V,0}: \vv(g)\ge\uupsilon\}
\end{equation}
for $\uupsilon=(\upsilon_1, \ldots, \upsilon_r)\in\ZZ_{\ge 0}^r$, $\vv(g)=(v_1(g), \ldots, v_r(g))$,
$\uupsilon'=(\upsilon'_1, \ldots, \upsilon'_r)\ge \uupsilon''=(\upsilon''_1, \ldots, \upsilon''_r)$
iff $\upsilon'_i\ge \upsilon''_i$ for all $i=1, \ldots, r$. 
(It is convenient to assume that the
equation (\ref{filtration}) defines the subspaces $J(\uupsilon)\subset \calO_{V,0}$ for all
$\uupsilon\in\ZZ^r$.) The {\em Poincar\'e series} $P_{\{v_i\}}(\ttt)$ ($\ttt=(t_1, \ldots, t_r)$) of the filtration (\ref{filtration}) can be defined as 
\begin{equation}\label{Poincare}
P_{\{v_i\}}(\ttt):=
\frac
{
\left(\sum_{\uupsilon\in\ZZ^r}\dim(J(\uupsilon)/J(\uupsilon+\1))\ttt^\uupsilon\right)\prod_{i=1}^r(t_i-1)
}
{(t_1t_2\cdots t_r-1)}\,,
\end{equation}
where $\1=(1,\ldots, 1)\in\ZZ^r$, $\ttt^\uupsilon=t_1^{\upsilon_1}\cdots t_r^{\upsilon_r}$
(see e.g. \cite{IJM}; it is defined when the dimensions of all the factor spaces $J(\uupsilon)/J(\uupsilon+\1)$
are finite). In \cite{IJM} it was explained that the Poincar\'e series (\ref{Poincare}) is equal
to the integral with respect to the Euler characteristic
\begin{equation}\label{integral}
P_{\{v_i\}}(\ttt)=\int_{\PP\calO_{V,0}}\ttt^{\vv(g)}d\chi
\end{equation}
over the projectivization $\PP\calO_{V,0}$ of the space $\calO_{V,0}$. (In the integral $t_i^{+\infty}$ has to be assumed to be equal to zero.)

Let $f\in\calO_{\CC^n,0}$ be a function germ with the Newton diagram $\Gamma=\Gamma_f\subset\RR^n$, $V:=\{f=0\}$. Let $\gamma_i$, $i=1, \ldots, r$, be (all) the facets of the diagram $\Gamma$ and let $\ell_i(\bar{k})=c_i$ be the reduced equation of the hyperplane containing the facet $\gamma_i$. One has $\ell_i(\bar{k})=\sum_{j=1}^n \ell_{ij}k_j$ ($\bar{k}=(k_1, \ldots, k_n)$), where $\ell_{ij}$ are positive integers, $\gcd{(\ell_{i1}, \ldots, \ell_{in})=1}$.

For $g\in {\calO}_{\CC^n,0}[x_1^{-1},\ldots, x_n^{-1}]$, $g=\sum\limits_{\bar{k}\in\ZZ^n}a_{\bar{k}}{\bar{x}}^{\bar{k}}$ ($\bar{x}=(x_1, \ldots, x_n)$), let
$$
u_i(g):=\min_{{\bar{k}}:a_{\bar{k}}\ne0}\ell_i({\bar{k}})\,.
$$
One can see that $u_i$ is a valuation on ${\calO}_{\CC^n,0}\subset {\calO}_{\CC^n,0}[x_1^{-1},\ldots, x_n^{-1}]$. For a Newton diagram $\Lambda$ in $\RR^n$, let 
$$
u_i(\Lambda):=\min_{\bar{k} \in \Lambda} \ell_i(\bar{k}).
$$
(It is also equal to $u_i(g)$ for any germ $g$ with the Newton diagram $\Lambda$.)
Let $g_{\gamma_i}(\bar{x}):=\sum\limits_{\bar{k}:\ell_i(\bar{k})=u_i(g)}a_{\bar{k}}{\bar{x}}^{\bar{k}}$.

The following two collections of order functions on ${\calO}_{\CC^n,0}$ corresponding to the pair $(\Gamma,f)$ were defined in \cite{EG1} and \cite{EG2} respectively:
\begin{eqnarray}
v'_i(g) & := & \sup_{h\in \calO_{\CC^n,0}} u_i(g+hf)\,, \\
v''_i(g) & := & \sup_{h\in \calO_{\CC^n,0}[x_1^{-1}, \ldots, x_n^{-1}]} u_i(g+hf)\,.
\end{eqnarray}
($v'_i$ and $v''_i$
are, in general, not valuations, at least when $n=2$ or when $f$ is degenerate with respect to its Newton diagram $\Gamma$.)
They can be considered as order functions on the ring  ${\calO}_{V,0}={\calO}_{\CC^n,0}/(f)$ as well. (These order functions and moreover the corresponding Poincar\'e series are, in general, different.) 

Assume that the function $f$ is non-degenerate with respect to its Newton diagram $\Gamma$ and let $p:(X,D)\to(\CC^n,0)$ be a toric resolution of $f$ corresponding to the Newton diagram $\Gamma$. The facets $\gamma_1$, \dots, $\gamma_r$ of $\Gamma$ correspond to some components (say, $E_1$, \dots, $E_r$) of the exceptional divisor $D$. Let $\widetilde{V}$ be the strict transform of the hypersurface singularity $V$ (it is a smooth complex manifold) and let $\calE_i:=\widetilde{V}\cap E_i$, $i=1, \ldots, r$.

For $n\geq 3$, the set $\calE_i$ is an irreducible component of the exceptional divisor
$\calD=D\cap\widetilde{V}$ of the resolution $p_{\vert\widetilde{V}}:(\widetilde{V}, \calD)\to (V,0)$.
The divisorial valuation $v_{\calE_i}$ on ${\calO}_{V,0}$ defined by this component coincides
with $v''_i$: see \cite{EG2}. For $n=2$, the set $\calE_i$ is, in general, reducible (if the integer
length $s_i$ of the facet (edge) $\gamma_i$ is greater than 1). Let
$\calE_i=\bigcup\limits_{j=1}^{s_i}\calE_i^{(j)}$ be the decomposition into the irreducible
components ($\calE_i^{(j)}$ are points on the curve $\widetilde V$). One can show that in this case
$v''_i(g)=\min\limits_j v_{\calE_i^{(j)}}(g)$,
where $v_{\calE_i^{(j)}}$ are the corresponding divisorial valuations on ${\calO}_{V,0}$. This order function $v_i''$ can be regarded as a generalized divisorial valuation.

\section{The Poincar\'e series}
Let $\Gamma$ be a Newton diagram in $\RR^2$ with the facets (edges) $\gamma_1$, \dots, $\gamma_r$ and let $f$ be a function germ $(\CC^2,0)\to (\CC, 0)$ with the Newton diagram $\Gamma$. One can see that $f=x^ay^b\prod\limits_{i=1}^r f_i$, where $f_i$ is such that
$f_{\gamma_i}=\lambda_i{\bar{x}}^{{\bar{k}}_i}(f_i)_{\gamma_i}$ for certain $\lambda_i \in \CC^*$ and $\bar{k}_i \in  \ZZ_{\ge 0}^2$. The Newton diagram $\Gamma_i$ of the germ $f_i$ consists of one segment congruent (by a shift; in particular, parallel) to the facet $\gamma_i$ with the vertices on the coordinate lines in
$\RR^2$.

Let $\MM_i=\uu(\Gamma_i)$, i.e.
$\MM_i=(M_{i1}, \ldots, M_{ir})$,
where $M_{ij}=u_j(\Gamma_i)$. 
(One can see that $\MM_i=s_i\mm_i$ in the notations of
\cite{EG2}.)

\begin{theorem} \label{theo1}
One has
\begin{equation} \label{main}
P_{\{v''_i\}}(\ttt)=\frac{\prod\limits_{i=1}^r(1-\ttt^{\MM_i})}{(1-\ttt^{\uu(x)})(1-\ttt^{\uu(y)})}\,.
\end{equation}
\end{theorem}

\begin{corollary}
For the number of variables $n=2$ the Poincar\'e series
$P_{\{v''_i\}}(\ttt)$ depends only on the Newton diagram $\Gamma$
and does not depend on $f$ with $\Gamma_f=\Gamma$.
\end{corollary}

For the proof of Theorem~\ref{theo1} we need some auxiliary statements. We first introduce some notation.

For a Newton diagram $\Lambda$ in $\RR^2$, let
$\Sigma_{\Lambda}$ be the corresponding Newton polygon:
$\Sigma_{\Lambda}=
\bigcup\limits_{\bar{q} \in\Lambda} \left( \bar{q} +\RR_{\ge0}^2 \right)$.
Let $\calO^{\Lambda}$ be the set of functions
$g\in\calO_{\CC^2,0}$ with the Newton diagram
$\Gamma_g=\Lambda$. For $\uupsilon\in\ZZ_{\ge0}^r$, let
$\calO^{\Lambda}_{\uupsilon}=
\{g\in \calO^{\Lambda}: {\vv}''(g)=\uupsilon\}$. 
The set $\calO_{\CC^2,0}\setminus\{0\}$ is the disjoint union
of the sets $\calO^{\Lambda}$ over all diagrams $\Lambda$. According to (\ref{integral}) one has

\begin{equation} \label{partition}
P_{\{v''_i\}} (\ttt) =
\sum_{\Lambda} \int_{\PP\calO^{\Lambda}}\ttt^{{\vv}''(g)}d\chi\,.
\end{equation}

For $g\in\calO^\Lambda$, one has $v''_i(g)\ge u_i(\Lambda)$.

We shall first show that the integrals in (\ref{partition}) can be restricted only to functions $g \in \PP\calO^{\Lambda}$ with $\vv''(g) = \uu(\Lambda)$.

\begin{proposition}\label{bad_functions}
For a Newton diagram $\Lambda$ in $\RR^2$, let
$\uupsilon\in\ZZ_{\ge0}^r$ be such that
$\uupsilon> \uu(\Lambda)$, i.e. $\upsilon_i\ge u_i(\Lambda)$ for all $i=1, \ldots, r$ and $\upsilon_i > u_i(\Lambda)$ for 
some $i$. Then the set
$\PP\calO^\Lambda_{\uupsilon}$ has Euler characteristic equal to zero.
\end{proposition}

\begin{remark} The direct analogue of this proposition does not hold for the filtration defined by the order functions $\{ v_i' \}$. As an example one can take $f(x,y)=y^5+xy^2+x^2y+x^5$ with the Newton diagram $\Gamma$ with the set of vertices $\{ (0,5), (1,2), (2,1), (5,0) \}$. One has $\ell_1(\bar{k})=3k_x+k_y$, $\ell_2(\bar{k}) = k_x+k_y$, $\ell_3(\bar{k})=k_x+3k_y$. Let $\Lambda$ be the Newton diagram with the set of vertices $\{ (0,5), (1,2)\}$. One has $\uu(\Lambda)=(5,3,7)$. Let us take $\uupsilon = (7,3,7)$. One can see that for the order functions $\{ v_i' \}$ the set $\calO^\Lambda_{\uupsilon}$ consists of the germs  $g(x,y)=\sum a_{ij}x^iy^j$ from $\calO^\Lambda$ with $a_{05}=a_{12} \neq 0$ and $a_{06}=a_{13}=0$. This gives $\chi(\PP\calO^\Lambda_{\uupsilon})=1$. For the order functions $\{ v_i'' \}$ the set $\calO^\Lambda_{\uupsilon}$ consists of the germs  with $a_{05}=a_{12} \neq 0$, $a_{06}=a_{13}=0$ and $a_{07} - a_{14}+a_{21} \neq 0$. This gives $\chi(\PP\calO^\Lambda_{\uupsilon})=0$ in accordance with Proposition~\ref{bad_functions}.
\end{remark}

For the proof of Proposition~\ref{bad_functions} we need two lemmas. 

Let $\calO^\Lambda$ be non-empty.
For $g \in \calO^\Lambda$ with $\vv''(g)=\uupsilon$ and for $i=1, \ldots, r$,
one can find $h(=h_i) \in \calO_{\CC^2,0}[x^{-1}, y^{-1}]$ such that the Newton diagram of $g+hf$ lies in the (closed) half-plane $H_i=\{\bar{k}:\ell_i(\bar{k})\ge \upsilon_i\}$,
but there are no $h$ for which the Newton diagram of $g+hf$ lies in the open half-plane $\{\bar{k}:\ell_i(\bar{k})> \upsilon_i\}$. Let $\Lambda^*$ be the union of the compact edges of the (infinite) polygon
$\Sigma_\Lambda^*=
\bigcap\limits_{i=1}^r H_i \cap \Sigma_\Lambda$, 
where $\Sigma_\Lambda$ is the Newton polygon corresponding to $\Lambda$. ($\Lambda^*$ is not, in general, 
a Newton diagram since it may have non-integral vertices. Nevertheless we shall use the name ``diagram" for it.)

\begin{lemma} \label{lemma1}
In the situation described above, there exists an index $i$ $(1\le i\le r)$ such that $\Lambda^*$ has an edge $\delta_i$ parallel to $\gamma_i$ and (strictly) longer than $\gamma_i$.
\end{lemma}

\begin{proof}
We shall prove that there exists an edge of the diagram $\Lambda^*$ which is (strictly) longer than the edge of $\Lambda$ parallel to it. This implies that this edge is parallel to a certain edge $\gamma_i$ of the diagram $\Gamma$ and is longer than it. 
Since we assumed $\calO^{\Lambda}$ being non-empty, all edges of $\Lambda^*$ are parallel to edges of $\Lambda$. Let $a_0 < a_1 < \ldots < a_\sigma$ be the $k_x$-coordinates of all the vertices of $\Lambda$. Let $b_0 \leq b_1 \leq \ldots \leq b_\sigma$ be the $k_x$-coordinates of the corresponding vertices of $\Lambda^*$, i.e.\ $[b_{i-1}, b_i]$ is the projection of the segment of $\Lambda^*$ parallel to the segment of $\Lambda$ projected to $[a_{i-1},a_i]$ ($b_{i-1}$ and $b_i$ may coincide). One can see that $a_0=b_0$ and $a_\sigma \leq b_\sigma$. Then either $b_i=a_i$ for all $i=0,1, \ldots , \sigma$ or $[a_{i-1},a_i] \subsetneq [b_{i-1}, b_i]$ for some $i \in \{ 0,1, \ldots, \sigma \}$. But the first case cannot happen since $\Lambda^*\neq \Lambda$.
\end{proof}

We shall also use the following generalized version of the division with remainder for Laurent polynomials.

\begin{lemma}\label{division}
Let $p(z)$ and $q(z)$ be Laurent polynomials in $z$. Assume that  ${\rm supp}\, q$ has length $s$, i.e.
$q(z)=\sum\limits_{i=0}^s b_iz^{d+i}$ with $b_0\ne 0$,
$b_s\ne 0$, and let $d'$ be an integer. Then the polynomial $p(z)$ has a unique representation of the form
$p(z)=q(z)a(z) + r(z)$ with $r(z)=\sum\limits_{i=0}^{s-1} c_iz^{d'+i}$ .
\end{lemma}

\begin{proof*}{Proof of Proposition~\ref{bad_functions}}
Let $i$ be as in Lemma~\ref{lemma1}.
Let the integer length of $\gamma_i$ be equal to $s_i$. Then the segment $\delta_i$ contains at least $s_i$ integer points. Let $Q_1, \ldots, Q_{s_i}$ be $s_i$ consecutive integer points on the segment $\delta_i$.
Let $g\in\calO^\Lambda$ be such that
$\vv''(g)=\uupsilon$ ($>\uu(\Lambda)$) and let
$\widetilde{g}=g+hf$ be a Laurent polynomial such that
${\rm supp}\,{\widetilde{g}}\subset
H_i=\{\ell_i(\bar{k})\ge \upsilon_i\}$. Lemma~\ref{division} implies that
$\widetilde{g}_{\gamma_i}(x,y)=f_{\gamma_i}(x,y)p_i(x,y) + r_i(x,y)$ where
${\rm supp}\,{r}\subset \{Q_1, \ldots, Q_{s_i}\}$ and $r_i\ne 0$ (otherwise $v_i(g)>\upsilon_i=u_i(\Lambda^*)$).
(Let us recall that ${\rm  supp}\, f_{\gamma_i}$ consists of $s_i+1$ consecutive points on the line containing $\gamma_i$.)
Moreover the polynomial $r_i$ depends only on $g$ and does not depend on the choice of $\widetilde{g}$ (i.e.\ on
the choice of $h$).

One can see that all functions $g'$ of the form $g+(\lambda -1)r_i$ with $\lambda \ne 0$ lie in
$\calO^{\Lambda}$ and satisfy the condition $v''_i(g')=v''_i(g)$. Thus the set $\PP\calO^{\Lambda}_{\uupsilon}$ is fibred by $\CC^*$-families and therefore its Euler characteristic is equal to zero.
\end{proof*}

Proposition~\ref{bad_functions} implies that 
\begin{equation}\label{int_good}
P_{\{ v''_i\}} (\ttt) =
\sum_{\Lambda} \int_{\PP\calO^{\Lambda}_{\uu(\Lambda)}}\ttt^{{\vv}''(g)}d\chi\,. 
\end{equation}

\begin{proposition}\label{bad_diagrams}
Suppose that a Newton diagram $\Lambda$ contains an edge $\delta$ not congruent to any edge of $\Gamma$, i.e. either not parallel to all the edges $\gamma_i$, or parallel to one of them, but of another length.
Then $\chi(\PP\calO^{\Lambda}_{\uu(\Lambda)})=0$.
\end{proposition}

\begin{proof}
Assume first that the edge $\delta$ is either not parallel to all the edges $\gamma_i$, or it is parallel to $\gamma_i$, but is shorter than $\gamma_i$. Let $\bar{q}=(q_x, q_y)$ and $\bar{q}'=(q'_x, q'_y)$, $q_x>q'_x$, be the vertices of the edge $\delta$ and let $\Lambda'$ be the set of points $\bar{k}$ in $\Lambda$ with $k_x\ge q_x$. A function germ
$g\in \calO^{\Lambda}_{\uu(\Lambda)}$ can be represented as $g_1+g_2$, where
${\rm supp}\,g_1\subset \Lambda'$, ${\rm supp}\,g_2\subset \Sigma_\Lambda\setminus\Lambda'$. (Note that $g_1\ne 0$ and $g_2\ne 0$.) One can see that all the functions of the form $g_1+\lambda g_2$ with $\lambda\ne 0$ lie in 
$\calO^{\Lambda}_{\uu(\Lambda)}$. Thus the set $\PP\calO^{\Lambda}_{\uu(\Lambda)}$ is fibred by $\CC^*$-families and therefore its Euler characteristic is equal to zero.

Now assume that the edge $\delta$ of the diagram $\Lambda$ is parallel to $\gamma_i$ and is longer than it.
Let $\bar{q}=(q_x, q_y)$ and $\bar{q}'=(q'_x, q'_y)$, $q_x>q'_x$, (respectively $\bar{q}_0=(q_{0x}, q_{0y})$ and
$\bar{q}'_0=(q'_{0x}, q'_{0y})$, $q_{0x}>q'_{0x}$) be the vertices of the edge $\delta$ (respectively of the edge
$\gamma_i$) and let $\Lambda'$ be defined as above: $\Lambda'=\{\bar{k}\in \Lambda: k_x\ge q_x\}$. Let
$g(\bar{x})=\sum a_{\bar{k}}{\bar{x}}^{\bar{k}} \in \calO^{\Lambda}_{\uu(\Lambda)}$,  $f(\bar{x})=\sum c_{\bar{k}}{\bar{x}}^{\bar{k}}$ ($\bar{x}=(x,y)$) and let
$$
g_1(\bar{x})=g_{\Lambda'}(\bar{x})-a_{\bar{q}}{\bar{x}}^{\bar{q}}+(a_{\bar{q}}/c_{\bar{q}_0})
f_{\gamma_i}(\bar{x})\cdot {\bar{x}}^{\bar{q}-\bar{q}_0},
$$
where $g_{\Lambda'}(\bar{x})=\sum\limits_{\bar{k}\in\Lambda'} a_{\bar{k}}{\bar{x}}^{\bar{k}}$,
$g_2=g-g_1$. One has ${\rm supp}\,g_1\subset \Lambda'\cup (\bar{q},\bar{q}')$,
${\rm supp}\,g_2\subset \Sigma_\Lambda\setminus\Lambda'$, $f_{\gamma_i} {\not|} \, (g_2)_{\gamma_i}$ (in $\calO_{\CC^2,0}[x^{-1}, x^{-1}]$), where $(\bar{q},\bar{q}')$ denotes the open line segment connecting the two points.
All the functions of the form $g_1+\lambda g_2$ with $\lambda\ne 0$ lie in 
$\calO^{\Lambda}_{\uu(\Lambda)}$. Thus the set $\PP\calO^{\Lambda}_{\uu(\Lambda)}$ is again fibred by $\CC^*$-families and therefore its Euler characteristic is equal to zero.
\end{proof}

\begin{proposition}\label{value}
Let the Newton diagram $\Lambda$ consist (only) of segments congruent to $\gamma_i$ for
$i\in I\subset \{1, \ldots, r\}$. Then $\chi(\PP\calO^{\Lambda}_{\uu(\Lambda)})=(-1)^{\# I}$.
\end{proposition}

\begin{proof}
For $I=\emptyset$, the statement is obvious. Let $I\ne\emptyset$. 
Let $\PP\calO^{\Lambda}_{i}$ be the set of functions $g\in\PP\calO^{\Lambda}_{\uu(\Lambda)}$ with
$f_{\gamma_i} | g_{\gamma_i}$ (in $\calO_{\CC^2,0}[x^{-1}, x^{-1}]$).
One has 
$$
\PP\calO^{\Lambda}_{\uu(\Lambda)}=
\PP\calO^{\Lambda}\setminus \bigcup_{i\in I}\PP\calO^{\Lambda}_{i}\,.
$$
Therefore
\begin{equation} \label{Euler}
\chi(\PP\calO^{\Lambda}_{\uu(\Lambda)})=
\chi(\PP\calO^{\Lambda})+
\sum_{I'\subset I, I'\ne\emptyset}(-1)^{\# I'}\chi \left(\bigcap_{i\in I'}\PP\calO^{\Lambda}_{i} \right)\,.
\end{equation}

Let $\bar{q}_i$, $i=0,1, \ldots , \# I$, be the vertices of the diagram $\Lambda$. The set $\PP\calO^{\Lambda}$ consists of functions $g(\bar{x})=\sum a_{\bar{k}} \bar{x}^{\bar{k}}$ with $a_{\bar{q}_i} \neq 0$ for $i=0,1, \ldots, \# I$ and $a_{\bar{k}} = 0$ for $\bar{k} \not\in \Sigma^{\Lambda}$. Its Euler characteristic is equal to zero. Assume that $I' \subsetneq I$, $I'  \neq \emptyset$. Let $[\bar{q}, \bar{q}']$ be an edge of $\Lambda$ congruent to $\gamma_i$, $i \in I \setminus I'$ ($\bar{q} = (q_x,q_y)$, $\bar{q}'=(q_x',q_y')$, $q_x>q_y'$). Let $\Lambda'= \{ \bar{k} \in \Lambda \, : \, k_x \geq q_x \}$. For $g \in \bigcap_{i \in I'} \PP\calO_i^{\Lambda}$, let $g_1(\bar{x})=g_{\Lambda'}(\bar{x})$, $g_2=g-g_1$. All the functions of the form $g_1+ \lambda g_2$ with $\lambda \neq 0$ belong to $\bigcap_{i \in I'} \PP\calO_i^{\Lambda}$. Thus $\bigcap_{i \in I'} \PP\calO_i^{\Lambda}$ is fibred by $\CC^\ast$-families and therefore $\chi(\bigcap_{i \in I'} \PP\calO_i^{\Lambda})=0$.

Let $f_I(\bar{x}):= \prod_{i \in I} f_i(\bar{x})$. The intersection $\bigcap_{i \in I} \PP\calO_i^{\Lambda}$ (the Euler characteristic of which corresponds to $I'=I$ in (\ref{Euler})) consists of the functions $g \in \PP\calO_{\CC^2,0}$ such that $g_\Lambda(\bar{x}) = \lambda \bar{x}^{\bar{a}} (f_I)_{\Gamma_{f_I}}$ where $\Gamma_{f_I}$ is the Newton diagram of $f_I$, $\lambda \neq 0$ and $\bar{x}^{\bar{a}}$ is a certain monomial. Therefore
\[ \chi \left(\bigcap_{i \in I} \PP\calO_i^{\Lambda} \right)=1. \]
\end{proof}

\begin{proof*}{Proof of Theorem~\ref{theo1}}
Propositions~\ref{bad_functions} and \ref{bad_diagrams} imply that 
\begin{equation} \label{finalPS}
P_{\{v_i'' \}}(\ttt)= \sum_\Lambda  \int_{\PP\calO^{\Lambda}_{\uu(\Lambda)}}\ttt^{{\vv}''(g)}d\chi
\end{equation}
where the sum runs over all diagrams $\Lambda$ consisting only of edges congruent to some of the edges $\gamma_i$ of the diagram $\Lambda$. Let the edges of $\Lambda$ be congruent to the edges $\gamma_i$ with $i \in I=I(\Lambda)$. Proposition~\ref{value} implies that the summand in (\ref{finalPS}) corresponding to such a diagram $\Lambda$ is equal to $(-1)^{\# I} \ttt^{\uu(\Lambda)}$.
All the diagrams of this sort are obtained from the diagrams $\Gamma_I=\Gamma_{f_I}$ by shifts by non-negative integral vectors $\bar{k}$, i.e.\ $\Lambda= \bar{k} + \Gamma_I$. One has $\uu(\Lambda)=\underline{\ell}(\bar{k}) + \sum_{i \in I} \MM_i$. Therefore
\begin{eqnarray*}
P_{\{v_i'' \}}(\ttt) & = & \sum_{\bar{k} \in \ZZ^2_{\geq 0}} \sum_{I \subset \{ 1, \ldots , r \}} (-1)^{\# I} \ttt^{\underline{\ell}(\bar{k}) + \sum_{i \in I} \MM_i} \\
& = & \left( \sum_{\bar{k} \in \ZZ^2_{\geq 0}} \ttt^{\underline{\ell}(\bar{k})} \right) \cdot \left( \sum_{I \subset \{ 1, \ldots , r \}} (-1)^{\# I} \ttt^{\sum_{i \in I} \MM_i} \right) \\
& = & \frac{\prod\limits_{i=1}^r(1-\ttt^{\MM_i})}{(1-\ttt^{\uu(x)})(1-\ttt^{\uu(y)})} .
\end{eqnarray*}
\end{proof*}

\section{Relation with an Alexander polynomial}
One can see that the equation (\ref{main}) gives the Poincar\'e series $P_{\{v''_i\}}(\ttt)$ as a finite product/ratio
of ``cyclotomic'' binomials of the form $(1-\ttt^\MM)$ with $\MM\in\ZZ_{> 0}^r$. This looks similar to the usual A'Campo
type expressions for monodromy zeta functions or for Alexander polynomials of algebraic links \cite{EN}.
Here we shall describe a relation between the Poincar\'e series (\ref{main}) and a certain Alexander polynomial.

A notion of the multi-variable Alexander polynomial for a finite collection of germs of functions
on $(\CC^n,0)$ was defined in \cite{S}: see Proposition 2.6.2 therein. (In \cite{S} it is called the (multi-variable) zeta function.) In a somewhat more precise form this definition can be found in \cite{GDC}. (The definition in \cite{GDC}
gives the one for a collection of functions if one considers the corresponding principal ideals.)

As above, let $\Gamma$ be a Newton diagram in $\RR^2$ with the edges $\gamma_1$, \dots, $\gamma_r$ of integer lengths
$s_1, \ldots, s_r$ and let $p:(X,D)\to(\CC^2, 0)$ be a toric modification of $(\CC^2, 0)$ corresponding to the diagram
$\Gamma$. For $i=1, \ldots, r$, let $\widetilde{C}_i$ be a germ of a smooth curve on $X$ transversal to the component $E_i$
of the exceptional divisor $D$. Let $C_i=p(\widetilde{C}_i)$ be the image of $\widetilde{C}_i$ in $(\CC^2, 0)$ and let
$L_i=C_i\cap S^3_\varepsilon$ be the corresponding knot in the 3 sphere $S^3_\varepsilon=S^3_\varepsilon(0)$ for $\varepsilon > 0$
small enough. The curve $C_i$ can be defined by an equation $g_i=0$ where $g_i$ is a function germ $(\CC^2, 0)\to(\CC, 0)$ with the Newton diagram
consisting of one segment parallel to $\gamma_i$, with the (integer) length 1 and with the vertices on the coordinate lines.

Let $\Delta_{\underline{g}}(\ttt)$ and $\Delta_{\underline{g}^{\underline{s}}}(\ttt)$ be the Alexander polynomials
of the collections of functions $\underline{g}=(g_1, \ldots , g_r)$ and
$\underline{g}^{\underline{s}}=(g_1^{s_1}, \ldots , g_r^{s_r})$ respectively. The polynomial $\Delta_{\underline{g}}(\ttt)$ is the classical Alexander polynomial $\Delta^L(\ttt)$ of the link $L = \bigcup L_i$ (see e.g.\ \cite{EN}).
A one-variable analogue of $\Delta_{\underline{g}^{\underline{s}}}(\ttt)$
is considered in \cite[I.5]{EN} as the Alexander polynomial of the multilink $L(\underline{s})=\bigcup s_iL_i$.
One has
$$
\Delta_{\underline{g}}(\ttt)=\frac{\prod_{i=1}^r (1-\ttt^{\mm_i})}{(1-\ttt^{\uu(x)})(1-\ttt^{\uu(y)})}\,,
$$
$$
\Delta_{\underline{g}^{\underline{s}}}(\ttt)=
\frac{\prod_{i=1}^r (1-\ttt^{{\underline{s}} \, \mm_i})}{(1-\ttt^{\uu(x)})(1-\ttt^{\uu(y)})}\,,
$$
where $\sss\, \mm_i=(s_1m_{1i}, s_2m_{2i}, \ldots, s_rm_{ri})$. The main result of \cite{IJM}
says that $\Delta_{\underline{g}}(\ttt)=\Delta^L(\ttt)$ coincides with the Poincar\'e series of the filtration corresponding to
the Newton diagram of the function $\prod_{i=1}^r g_i$.

Let the {\em reduced Poincar\'e series} of the filtration defined by $\{ v_i'' \}$ be
$$
\widetilde{P}_{\{ v_i''\}} ( \ttt) := P_{\{ v_i''\}} ( \ttt) / P_{\{ u_i\}}(\ttt), \quad P_{\{ u_i\}}(\ttt)= \frac{1}{(1-\ttt^{\uu(x)})(1-\ttt^{\uu(y)})}
$$
being the Poincar\'e series of the filtration defined by the quasihomogeneous valuations $\{ u_i \}$ on $\calO_{\CC^2,0}$. One has 
\begin{equation} \label{redPS}
\widetilde{P}_{\{ v_i''\}} ( \ttt) = \prod\limits_{i=1}^r (1- \ttt^{s_i \mm_i}).
\end{equation}

Let 
$$
\widetilde{\Delta}_{\gggg^{\sss}} ( \ttt) := \Delta_{\gggg^{\sss}} ( \ttt) / \Delta^x_{\gggg^{\sss}} ( \ttt) \cdot \Delta^y_{\gggg^{\sss}} ( \ttt)
$$ 
where $\Delta^x_{\gggg^{\sss}} ( \ttt)$ and $\Delta^y_{\gggg^{\sss}} ( \ttt)$ are the Alexander polynomials of the sets of functions $\gggg^{\sss}= \{ g_1^{s_1}, \ldots , g_r^{s_r} \}$ restricted to the coordinate axes $\CC_x$ and $\CC_y$ respectively. One can regard $\widetilde{\Delta}_{\gggg^{\sss}} ( \ttt)$ as the Alexander polynomial of the set  of functions $\gggg^{\sss}$ restricted to the complex torus $(\CC^*)^2 \subset \CC^2$. One has
\begin{equation} \label{Alexander}
\widetilde{\Delta}_{\gggg^{\sss}} ( \ttt) = \prod\limits_{i=1}^r (1- \ttt^{\sss\, \mm_i}).
\end{equation}
One can see that a relation between (\ref{redPS}) and (\ref{Alexander}) can be described in the following way. Consider products of $r$ ordered cyclotomic binomials in $r$ variables. Such a product 
$$
\prod\limits_{i=1}^r (1-\ttt^{\NNN_i}), \quad \NNN_i=(N_{i1}, \ldots , N_{ir}),
$$
can be described by the corresponding $r \times r$-matrix $N:=(N_{ij})$. The transposition of the matrix induces an involution on the set of products of this sort. One can see that this involution maps the product (\ref{redPS}) for the Poincar\'e series to the product (\ref{Alexander}) for the Alexander polynomial.


\bigskip
\noindent Leibniz Universit\"{a}t Hannover, Institut f\"{u}r Algebraische Geometrie,\\
Postfach 6009, D-30060 Hannover, Germany \\
E-mail: ebeling@math.uni-hannover.de\\

\medskip
\noindent Moscow State University, Faculty of Mechanics and Mathematics,\\
Moscow, GSP-1, 119991, Russia\\
E-mail: sabir@mccme.ru


\begin{thebibliography}{1}

\bibitem{IJM} A.~Campillo, F.~Delgado, S.~M.~Gusein-Zade: The Alexander polynomial of a plane curve singularity and integrals with respect to the Euler characteristic. International Journal of Mathematics {\bf 14}, no.1, 47--54 (2003).


\bibitem{EG1} W.~Ebeling, S.~M.~Gusein-Zade: Multi-variable Poincar\'e series associated with Newton diagrams. Journal of Singularities {\bf 1}, 60--68 (2010).

\bibitem{EG2} W.~Ebeling, S.~M.~Gusein-Zade: On divisorial filtrations associated with Newton diagrams. Journal of Singularities {\bf 3}, 1--7 (2011).
 
\bibitem{EN} D.~Eisenbud, W.~Neumann: Three-dimensional link theory and invariants of plane curve singularities. Annals of Mathematics Studies, 110. Princeton University Press, Princeton, NJ, 1985.

\bibitem{GDC} S.~M.~Gusein-Zade, F.~Delgado, A.~Campillo: Alexander polynomials and Poincar\'e series of sets of ideals.  Funktsional. Anal. i Prilozhen. {\bf 45}, no.4, 40--48 (2011) (Engl. translation in Funct. Anal. Appl. {\bf 45}, no.4, 271--277 (2011)).

\bibitem{H} H.~Hamm: On the Newton filtration for functions on complete intersections. arXiv: 1103.0654.

\bibitem{L} A.~Lemahieu: Poincar\'e series of embedded filtrations. Math. Res. Lett. {\bf 18}, no.5, 815--825 (2011).

\bibitem{S} C.~Sabbah: Modules d'Alexander et $\cal D$-modules. Duke Math. J. {\bf 60}, no.3, 729--814 (1990).

\end{thebibliography}
\end{document}